\documentclass{article}
\pdfoutput=1
\author{Daniel Pryor}
\title{Special Open Sets in Manifold Calculus}

\usepackage{amsmath}
\usepackage{amsfonts}
\usepackage{amsthm}
\usepackage{amssymb}
\usepackage[capitalize]{cleveref}
\usepackage{xypic}
\usepackage{tensor}

\theoremstyle{plain}
\newtheorem{prop}{Proposition}[section]
\newtheorem{lemma}[prop]{Lemma}
\newtheorem{cor}[prop]{Corollary}
\newtheorem{thm}[prop]{Theorem}

\theoremstyle{definition}
\newtheorem{defn}[prop]{Definition}

\newtheorem{notation}[prop]{Notation}

\newtheorem{rmk}[prop]{Remark}
\newtheorem{example}[prop]{Example}

\newcommand{\abs}[1]{\left|#1\right|}

\newcommand{\parens}[1]{\left(#1\right)}

\newcommand{\bracks}[1]{\left[#1\right]}

\newcommand{\braces}[1]{\left\{#1\right\}}

\newcommand{\reals}{\mathbb{R}}

\newcommand{\conj}[1]{\overline{#1}}

\newcommand{\vphi}{\varphi}

\newcommand{\mc}[1]{\mathcal{#1}}
\newcommand{\wh}[1]{\widehat{#1}}

\newenvironment{enumabc}
{\begin{enumerate}}
{\end{enumerate}}

\DeclareMathOperator{\Tot}{Tot}
\DeclareMathOperator{\crep}{crep}
\DeclareMathOperator{\srep}{srep}
\DeclareMathOperator*{\holim}{holim}
\DeclareMathOperator*{\hocolim}{hocolim}

\DeclareMathOperator{\Map}{Map}
\DeclareMathOperator{\mps}{mps}

\DeclareMathOperator{\Emb}{Emb}
\DeclareMathOperator{\Imm}{Imm}
\DeclareMathOperator{\Aut}{Aut}
\DeclareMathOperator{\const}{const}

\newcommand{\ctgy}[1]{\mathbf{#1}}

\begin{document}
\maketitle
\bibliographystyle{plain}
\begin{abstract}
Embedding Calculus, as described by Weiss, is a calculus of functors,
suitable for studying contravariant functors from the poset of open
subsets of a smooth manifold $M$, denoted $\mc{O}(M)$, to a category
of topological spaces (of which the functor $\Emb(-,N)$ for some fixed
manifold $N$ is a prime example).  Polynomial functors of degree $k$
can be characterized by their restriction to $\mc{O}_k(M)$, the full
subposet of $\mc{O}(M)$ consisting of open sets which are a disjoint
union of at most $k$ components, each diffeomorphic to the open unit
ball.  In this work, we replace $\mc{O}_k(M)$ by more general subposets
and see that we still recover the same notion of polynomial
cofunctor.
\end{abstract}

\tableofcontents
\section{Introduction}
In \cite{We}, Weiss develops manifold calculus, a variation
on Goodwillie's calculus of homotopy functors in \cite{Go2}.
Manifold calculus studies contravariant topological space-valued
functors on the poset of open subsets of a manifold $M$.  Manifold
calculus is especially good for studying spaces of smooth embeddings
of one manifold into another by looking at the functor $\Emb(-,N)$ for
a fixed manifold $N$, which is the apparent motivation behind \cite{We}.
The main goal of this work is to generalize Weiss' characterization of
polynomial cofunctors.

Being a calculus of functors, manifold calculus has a notion
of polynomial cofunctor.  These are the cofunctors which
satisfy an appropriate higher-order excision property, similar
to the case of \cite{Go2}.  Weiss is able to characterize degree
$k$ polynomial cofunctors as follows.  Let $\mc{O}$ be the poset
of open sets of a $d$-dimensional manifold $M$, and let $\mc{O}_k$
be the full subposet of $\mc{O}$ whose objects are disjoint unions
of at most $k$ components, each diffeomorphic to $\reals^d$.  Weiss
calls objects of $\mc{O}_k$ special open sets.  Then
a degree $k$ polynomial cofunctor $F:\mc{O}\to\ctgy{Top}$ is
determined (up to equivalence) by its restriction to $\mc{O}_k$.
The main result in this paper (\cref{thm:BkRefinement})
is a statement that generalizes this characterization of a
$k$-polynomial cofunctor by its restriction to a subposet
$\mc{B}_k$ of $\mc{O}_k$.  The objects of $\mc{B}_k$ are simply
disjoint unions of the objects of $\mc{B}_1$.  As long as the
objects of $\mc{B}_1$, form a basis for the topology of $M$,
then no homotopy theoretic information is lost when forming
the polynomial approximation to a cofunctor using $\mc{B}_k$
instead of $\mc{O}_k$  (\cref{thm:PolyUniqueness} and
\cref{thm:FBangPoly}).

We now give a brief outline of this work.  In \cref{sec:Conventions},
we quickly go over some of the conventions and basic notions from
homotopy theory that we will need.  Then, in \cref{sec:MfldCalc} and
\cref{sec:PolyCofunctors}, we briefly introduce Mainfold Calculus,
summarizing some of the main results of Weiss.  We define $\mc{O}_k(M)$,
the special open sets of Weiss mentioned earlier.  We recall Weiss'
construction of polynomial cofunctors as suitable extensions of
cofunctors defined only on $\mc{O}_k(M)$.

In \cref{sec:CatLemmas}, we recall and prove some fairly general
results about functors from any category $\mc{C}$ to $\ctgy{Top}$.
Then, in \cref{sec:SpecialOpenSets}, we generalize the notion of
special open set as mentioned earlier.  Namely, we let $\mc{B}_k(M)$
be full subposets of $\mc{O}_k(M)$ which still contain enough open
sets so that we lose no information when we develop the analogous
theory.  The primary example of interest here occurs when $M$ is a
smooth codimension zero submanifold of $\reals^d$, and $\mc{B}_k(M)$
contains all open sets which are disjoint unions of at most $k$ open
balls (in the euclidean metric sense).

We prove the analogs of several results of Weiss in this more
general setting, adding in a few details as well.  Most proofs go
through with very little change, but one notable exception is
\cref{thm:WeissPolyUniqueness}, where having all of $\mc{O}_k$
as special open sets is crucial to his proof.  Then we prove the
main result (\cref{thm:BkRefinement}) that any valid choice
(as described above) of special open sets yields equivalent
notions of polynomial cofunctors and polynomial approximation
by polynomial cofunctors.

\textbf{Acknowledgements}.  This work comes from part of my doctoral
thesis, which was written at the University of Virginia under
the supervision of Gregory Arone, to whom I am very grateful for
his continued encouragement and advice.

\section{Conventions}\label{sec:Conventions}
In this section, we introduce some conventions and recall some basic
notions from homotopy theory.  We work in the category of weak Hausdorff
compactly generated topological spaces, wich we denote $\ctgy{Top}$.
For topological spaces $X$ and $Y$, we let $\Map(X,Y)$ denote the
space of maps (continuous functions) from $X$ to $Y$, with the
compact-open topology.

If $p:E\to B$ is a map of (unbased) topological spaces, its
mapping path space, $\mps(p)$, is defined to be
the subspace
\begin{equation*}
  \mps(p) = \braces{(e,\vphi) \in E \times \Map(I,B) \mid
  \vphi(0) = p(e)}
\end{equation*}
of $E \times \Map(I,B)$.  Any such map can be factored as
\begin{equation*}
  E \to \mps(p) \to B.
\end{equation*}
The lefthand map sends a point $e\in E$ to the point
$(e,\const_{p(e)})$, where $\const_{p(e)}$ denotes the constant
path at $p(e)$; this map is a homotopy equivalence.  The
righthand map sends a point $(e,\vphi) \in \mps(p)$ to the point
$\vphi(1) \in B$; this map is a fibration.

Let $\mc{C}$ be a small category, and suppose $F:\mc{C}\to\ctgy{Top}$
is a functor.  Along the lines of \cite{BoKa}, we write $\srep(F)$ to
denote the simplicial replacement of $F$, a simplicial space whose
geometric realization gives the homotopy colimit of $F$.  Dually, we
let $\crep(F)$ be the cosimplicial replacement of $F$, a
cosimplicial space whose totalization gives the homotopy limit
of $F$.  In this work, goemetric realization and totalization will
mean homotopy invariant geometric realization and totalization.

\section{Manifold Calculus Preliminaries}\label{sec:MfldCalc}
Let $M$ be a $d$-dimensional smooth manifold without boundary.
Define $\mc{O} = \mc{O}(M)$ to be the poset of open subsets of $M$,
considered as a category.  If $V$ is an open subset of $M$, let
$\mc{O}(V)$ be the full subposet of $\mc{O}$ whose objects are
contained in $V$.  Equivalently, we can think of $\mc{O}(V)$
as the comma category $(\mc{O}(M) \downarrow V)$.

\begin{defn}
A cofunctor $F$ from a subcategory of $\mc{O}$ to $\ctgy{Top}$
will be called an isotopy cofunctor if $F$ takes all
inclusions which are isotopy equivalences to homotopy equivalences.
A cofunctor $F:\mc{O}\to\ctgy{Top}$ will be called good
if it satisfies the following conditions:
\begin{enumabc}
  \item $F$ is an isotopy cofunctor.
  \item If $V_0 \to V_1 \to \cdots$ is a string of inclusions in
  $\mc{O}$, then the natural map
  \begin{equation*}
    F\parens{\bigcup_{i=0}^\infty V_i} \to \holim_i F(V_i)
  \end{equation*}
  is a homotopy equivalence.
\end{enumabc}
\end{defn}

\begin{example}
For a fixed manifold $N$, Proposition 1.4 in \cite{We} shows that
the cofunctors $\Emb(-,N)$ and $\Imm(-,N)$ are good, where
$\Emb$ and $\Imm$ denote the spaces of smooth embeddings and
immersions, respectively.
\end{example}

Let $V$ be an open subset of $M$, and let $C_0,\ldots,C_k$ be
pairwise disjoint closed subsets of $V$.  For $S \subseteq
\braces{0,\ldots,k}$, let
\begin{equation*}
V_S = V \setminus \bigcup_{i\in S} C_i = \bigcap_{i\in S}
  \parens{V \setminus C_i}.
\end{equation*}
We thus have a $(k+1)$-cube of spaces $S \mapsto F(V_S)$.
We say that this cube is homotopy cartesian if the natural map
\begin{equation*}
  F(V) \to \holim_{S\neq\emptyset} F(V_S)
\end{equation*}
is a homotopy equivalence.  For example, if $k = 1$ this becomes
the requirement that
\begin{equation*}
\xymatrix
{
  F(V = V_\emptyset) \ar[r] \ar[d] &
  F(V_{\braces{0}}) \ar[d] \\
  F(V_{\braces{1}}) \ar[r] &
  F(V_{\braces{0,1}}) \\
}
\end{equation*}
is a homotopy pullback square.  See \cite{Go1} for more background
on cubical diagrams of spaces.

\begin{defn}
A good cofunctor $F$ is \textbf{polynomial} of degree $\leq k$
if for any choices of $V$ and $C_0,\ldots,C_k$, the $(k+1)$-cube
$S\mapsto F(V_S)$ is homotopy cartesian.
\end{defn}

\begin{example}
Example 2.3 in \cite{We} shows that $\Imm(-,N)$ is linear (polynomial
of degree $\leq 1$) if $\dim M < \dim N$, or if $\dim M = \dim N$ and
$M$ has no compact components.
\end{example}

\section{Polynomial Cofunctors and the Taylor Tower}\label{sec:PolyCofunctors}
Let $\mc{O}_k$ be the full subposet of $\mc{O}$
consisting of those open sets which are diffeomorphic to
a disjoint union of at most $k$ copies of $\reals^d$.
Weiss calls these special open sets and shows \cite{We} that
$k$-polynomial cofunctors are determined by their restriction
to $\mc{O}_k$.  More precisely, we have the following.

\begin{thm}[\cite{We}, 5.1]\label{thm:WeissPolyUniqueness}
If $F,G:\mc{O}\to\ctgy{Top}$ are polynomial cofunctors of degree
$\leq k$, and $\gamma:F\to G$ is a natural transformation such
that $\gamma_V:F(V)\to G(V)$ is a homotopy equivalence for
all $V \in \mc{O}_k$, then $\gamma_V$ is a homotopy equivalence
for all $V \in \mc{O}$.
\end{thm}

Moreover, Weiss shows \cite{We} that any (isotopy) cofunctor 
$\mc{O}_k\to\ctgy{Top}$ has a canonical extension to a $k$-polynomial
cofunctor $\mc{O}\to\ctgy{Top}$.  We introduce some notation
and then give the result.

\begin{notation}
Let $\mc{N}$ be a subposet of $\mc{O}$.  If $F$ is a cofunctor
from $\mc{N}\to\ctgy{Top}$, then define $F^!:\mc{O}\to\ctgy{Top}$
to be the homotopy right Kan extension of $F$ along the inclusion
functor $\mc{N}\to\mc{O}$.  An explicit formula given on objects is:
\begin{equation*}
  F^!(V) = \holim_{U\in\mc{N}(V)} F(U).
\end{equation*}
\end{notation}

\begin{thm}[\cite{We}, 3.8 and 4.1]\label{thm:WeissFBangPoly}
If $F:\mc{O}_k \to \ctgy{Top}$ is an isotopy cofunctor, then
$F^!:\mc{O} \to \ctgy{Top}$ is (good and) polynomial of degree $\leq k$.
\end{thm}

Combining these two results leads to the notion of the polynomial
approximation of a cofunctor.  Specifically, for $F:\mc{O}(M)\to\ctgy{Top}$
a good cofunctor, Weiss defines the degree $k$ polynomial approximation
to $F$, written $T_k F$, by the formula $T_k F = (F\vert_{\mc{O}_k})^!$.
That is, $T_k F$ is obtained by first restricting $F$ to $\mc{O}_k(M)$,
then by extending it back to all of $\mc{O}(M)$.  Thus, $T_k$ is an
endofunctor on the category of good cofunctors $\mc{O}(M)\to
\ctgy{Top}$.

The polynomial approximations to $F$ fit into a tower:
\begin{equation*}
\xymatrix
{
  & \vdots \ar[d] \\
  & T_2F \ar[d] \\
  & T_1F \ar[d] \\
  F \ar[uur] \ar[ur] \ar[r] & T_0F \\
}
\end{equation*}
called the Taylor tower of $F$.  The vertical maps $T_k F \to
T_{k-1}F$ come from applying $T_k$ to the natural map
$F \to T_{k-1} F$ and then observing that $T_k T_{k-1}$ is naturally
equivalent to $T_{k-1}$.

\section{Categorical Lemmas}\label{sec:CatLemmas}
\begin{defn}
A map $p:E\to B$ is a quasifibration if for all $e\in E$, the
canonical inclusion of the fiber of $p$ over $p(e)$ into the
homotopy fiber of $p$ over $p(e)$ is a weak homotopy equivalence.
\end{defn}

\begin{notation}
If $\mc{C}$ is a category, then we write $\abs{\mc{C}}$ for
the classifying space of $\mc{C}$, that is, for the geometric
realization of the nerve of $\mc{C}$.
\end{notation}

\begin{thm}[Quillen-Dwyer Theorem]\label{thm:QuillenDwyer}
Let $\mc{C}$ be a small category, and let $F:\mc{C}\to\ctgy{Top}$ be
a functor which sends all morphisms to homotopy equivalences.
Then the projection from $\hocolim F \to \abs{\mc{C}}$ is a
quasifibration.  Moreover, if $p':\mps(p)\to\abs{\mc{C}}$ is
the associated fibration, then $\holim F$ is equivalent to
the space of sections of $p'$.
\end{thm}
The quasifibration statement of this result is due to Quillen
in \cite{Qu}, and the identification of the homotopy limit
of $F$ with the space of sections of the associated
fibration is due to Dwyer in \cite{Dw}.

\begin{lemma}\label{thm:RestrictedHolimEquiv}
Let $\mc{D}$ be a small category, and let $\mc{C}$ be a
subcategory of $\mc{D}$.  Let $F:\mc{D}\to\ctgy{Top}$ be a
functor taking all morphisms to homotopy equivalences.  If the
inclusion $\abs{\mc{C}} \hookrightarrow \abs{\mc{D}}$ is a
homotopy equivalence, then so is the map
\begin{equation*}
\holim_\mc{D} F \to \holim_\mc{C} F\vert_\mc{C}.
\end{equation*}
\end{lemma}
\begin{proof}
Let $p:\hocolim_\mc{C} F\vert_{\mc{C}} \to \abs{\mc{C}}$ and $q:
\hocolim_{\mc{D}} F \to \abs{\mc{D}}$ be the projections.  Consider
the diagram
\begin{equation*}
  \xymatrix
  {
    \hocolim_\mc{C} F\vert_\mc{C} \ar[r]^-\simeq \ar[d] &
      \mps(p) \ar[r]^{p'} \ar[d] &
      \abs{\mc{C}} \ar[d]\\
    \hocolim_\mc{D} F \ar[r]^-\simeq &
      \mps(q) \ar[r]^{q'} &
      \abs{\mc{D}}. \\
  }
\end{equation*}
By \cref{thm:QuillenDwyer}, $p$ and $q$ are quasifibrations, and in this
diagram, we have (functorially) factored these maps through their mapping
path spaces as homotopy equivalences followed by fibrations.  The map of
homotopy limits will be a homotopy equivalence if the associated map of
section spaces, $\Gamma_{q'} \to \Gamma_{p'}$.  This will be the case if
both of the right two vertical maps are homotopy equivalences.  The
righthand one is by hypothesis, and the middle one will be if the
lefthand one is.

Choose a basepoint in $\abs{\mc{C}}$, a 0-simplex corresponding
to some object $c$.  Then the fiber of the map
\begin{equation*}
  p:\hocolim_\mc{C} F\vert_\mc{C} \to \abs{\mc{C}} \simeq \hocolim_\mc{C} *
\end{equation*}
can be taken to be just $F(c)$.  See this by looking at the induced
map of simplicial replacements, then observing that we get the
constant simplicial space $F(c)$ as the fiber.  So now we have a
map of quasifibration sequences
\begin{equation*}
  \xymatrix
  {
    F(c) \ar[r] \ar[d] & \hocolim_\mc{C} F\vert_\mc{C} \ar[r]^-p \ar[d] &
      \abs{\mc{C}} \ar[d] \\
    F(c) \ar[r] & \hocolim_\mc{D} F \ar[r]^-q & \abs{\mc{D}} \\
  }
\end{equation*}
Again, the righthand map is a homotopy equivalence, and the
lefthand map can be assumed to be the identity (again looking at
simplicial replacements).  The choice of basepoint was arbitrary
(arbitrary enough, as we have at least one for each path component of
$\abs{\mc{C}}$), so the middle map is also a homotopy equivalence,
and the lemma is proved.
\end{proof}

\begin{lemma}\label{thm:FilteredHolim}
Let $\mc{C}_0 \hookrightarrow \mc{C}_1 \hookrightarrow \mc{C}_2
\hookrightarrow \cdots$ be an increasing inclusion of small
categories, and call their union $\mc{C}$.  Let $F:\mc{C}\to
\ctgy{Top}$ be a functor and denote its restriction to $\mc{C}_i$
by $F_i$.  Then the natural map
\begin{equation*}
\holim_\mc{C} F \to \holim_i \holim_{\mc{C}_i} F_i,
\end{equation*}
is a homotopy equivalence.
\end{lemma}
\begin{proof}
We can view this as a map of totalizations of cosimplicial spaces:
\begin{equation*}
\Tot(\crep F) \to \holim_i \Tot(\crep F_i).
\end{equation*}
But totalization commutes with homotopy limits (since for
example, $\Tot X \simeq \holim_{\mathbf{\Delta}} X$ by \cite{BoKa}).
So this becomes the map
\begin{equation*}
\Tot(\crep F) \to \Tot(\holim_i \crep F_i).
\end{equation*}
On the righthand side, $\crep F_i$ is the diagram
\begin{equation*}
\xymatrix{
\crep(F_0)^0 \ar@2{->}@<3pt>[r] & \crep(F_0)^1 \ar@3{->}@<3pt>[r]
  \ar@<3pt>[l] & \cdots \ar@2{->}@<3pt>[l] \\
\crep(F_1)^0 \ar@2{->}@<3pt>[r] \ar[u] & \crep(F_1)^1
  \ar@3{->}@<3pt>[r] \ar@<3pt>[l] \ar[u] & \cdots \ar@2{->}@<3pt>[l]\\
\cdots \ar[u] & \cdots \ar[u] & \\
}\end{equation*}
So, after interchanging the totalization with the homotopy
limit, we are now taking homotopy limits in the vertical direction
first, then the totalization of the resulting cosimplicial
space second.  But the vertical maps are all fibrations
(projections onto subproducts in fact), so we can just take
ordinary limits.  But note that in codegree $q$, this limit
is just $\crep(F)^q$.
\end{proof}

Let $\mc{C}$ be a small category, and let $F:\mc{C}\to\ctgy{Cat}$
be a functor.  The Grothendieck construction is a new category
$\mc{C}\int F$ (alternatively $F \rtimes \mc{C}$) whose objects are
pairs $(c,x)$ where $c$ is an object of $\mc{C}$, and $x$ is
an object of $F(c)$.  A morphism $(c,x) \to (c',x')$ is a pair
$(f,g)$, where $f:c\to c'$ is a morphism in $\mc{C}$, and
$g:F(f)(x) \to x'$ is a morphism in $F(c')$.  If $(f,g):(c,x)\to
(c',x')$ and $(f',g'):(c',x')\to (c'',x'')$ are morphisms,
the composite $(f',g')\circ (f,g)$ is defined to be
\begin{equation*}
  (f',g') \circ (f,g) = (f'\circ f, g' \circ F(f')(g)):(c,x)\to(c'',x'').
\end{equation*}

\begin{example}
If $\mc{C}$ is a group $G$ (as a category with one object), and $F$
takes the object of $\mc{C}$ to another group $H$ in $\ctgy{Cat}$,
then we can consider $F$ as a group homomorphism $\vphi:G\to\Aut(H)$.
In this case, the Grothendieck construction recovers the usual
semidirect product $H \rtimes_\vphi G$.
\end{example}

Thomason \cite{Th} proves the following.
\begin{thm}[Thomason's homotopy colimit theorem]\label{thm:ThomasonHocolim}
For $\mc{C}$ and $F$ as above, there is a natural homotopy equivalence
\begin{equation*}
  \abs{\mc{C}\int F} \simeq \hocolim_{c\in\mc{C}} \abs{F(c)}
\end{equation*}
\end{thm}

\section{Special Open Sets}\label{sec:SpecialOpenSets}
We are now ready to tackle our main goal.  Namely, we would
like to characterize $k$-polynomial cofunctors by their restriction
to some smaller class of special open sets than all of $\mc{O}_k$.
We can do this provided that we leave enough special open sets in
place.  More precisely, for each $k \geq 0$, let $\mc{B}_k$ be a full
subposet of $\mc{O}_k$ satisfying each of the following conditions:
\begin{enumabc}
\item The objects of $\mc{B}_1$ form a basis for the topology of $M$.
\item The objects of $\mc{B}_k$ are exactly those which are a union
  of at most $k$ pairwise disjoint objects of $\mc{B}_1$.
\end{enumabc}
Note that the second condition implies that each $\mc{B}_k$ contains
the empty set as one of its objects.  Also note that once $\mc{B}_1$
is chosen, the rest of the $\mc{B}_k$ are determined automatically.
Now let $\mc{A}_k$ be the wide subposet (same objects, but possibly fewer
morphisms) of $\mc{B}_k$ with the weaker order where $U \leq V$ if the
inclusion of $U$ into $V$ is an isotopy equivalence.

\begin{example}
One possible choice for $\mc{B}_k$ is $\mc{O}_k$ itself.
This is the case that Weiss considers in \cite{We}, where he uses
the notation $\mc{O}k$ instead.  In this example, our $\mc{A}_k$ is
exactly Weiss' $\mc{I}k$.
\end{example}

\begin{example}
If $M$ is given as a smooth codimension zero submanifold
of $\reals^d$, then we can take $\mc{B}_k$ to be the subsets of
$M$ which are unions of at most $k$ pairwise disjoint open balls
(with respect to the euclidean metric), or cubes, simplices,
or convex $d$-bodies more generally.
\end{example}

First, we can very easily strengthen Weiss' characterization of polynomial
cofunctors by considering their restriction to $\mc{B}_k$ as follows.

\begin{cor}\label{thm:PolyUniqueness}
Let $F_1$ and $F_2$ be good cofunctors from $\mc{O} \to
\ctgy{Top}$, both polynomial of degree $\leq k$.  If $\gamma:F_1
\to F_2$ is a natural map such that $\gamma_V:F_1(V) \to F_2(V)$ is
a homotopy equivalence for all $V \in \mc{B}_k$, then it is a homotopy
equivalence for all $V \in \mc{O}$.
\end{cor}
\begin{proof}
Let $V \in \mc{O}_k$.  Then there is a $U \in \mc{B}_k$ such that
the inclusion of $U$ into $V$ is an isotopy equivalence.  We get
a commutative square,
\begin{equation*}
\xymatrix
{
  F_1(V) \ar[r] \ar[d]_{\gamma_V} & F_1(U) \ar[d]^{\gamma_U}\\
  F_2(V) \ar[r] & F_2(U).
}
\end{equation*}
The top and bottom maps are induced by isotopy equivalences, so
are themselves homotopy equivlances since $F_1$ and $F_2$ are
good.  And $\gamma_U$ is a homotopy equivalence by assumption,
so therefore $\gamma_V$ is a homotopy equivlance as well.
Then \cref{thm:WeissPolyUniqueness} implies that $\gamma_V$ is in
fact a homotopy equivalence for all $V \in \mc{O}$.
\end{proof}

\begin{rmk}
The rest of the development of this section mimics that of Weiss,
but in slightly more generality.  Still, this development is logically
independent from \cite{We}.  However, the previous characterization
of polynomial cofunctors appealed directly to Weiss' result.
His proof relied on using all of $\mc{O}_k$ in a critical way;
it does not seem that the argument can be made valid when using
some arbitrary choice for $\mc{B}_k$.
\end{rmk}

\begin{notation}
Let $X$ be a topological space.  Let $C(X,k)$ denote the (ordered)
configuration space of $k$ points in $X$.  That is,
\begin{equation*}
  C(X,k) = \braces{(x_1,\ldots,x_k) \in X^k \mid x_i \neq x_j
  \text{ for } i \neq j},
\end{equation*}
with the subspace topology.  Let $\binom{X}{k}$ denote the unordered
configuration space of $k$ points in $X$.  That is, the quotient of
$C(X,k)$ by the (free) action of the symmetric group $\Sigma_k$ which
permutes coordinates.
\end{notation}

The following result is the analog of Lemma 3.5 in \cite{We}.
This proof is based on that of Weiss, but treats the case for
general $j$ all at once.

\begin{prop}\label{thm:NerveAk}
\begin{equation*}
  \abs{\mc{A}_k(M)} \simeq \coprod_{j=0}^k \binom{M}{j}.
\end{equation*}
\end{prop}
\begin{proof}
First, note that $\mc{A}_k$ is a disjoint union
\begin{equation*}
\mc{A}_k = \coprod_{j=0}^k \mc{A}^{(j)},
\end{equation*}
where $\mc{A}^{(j)}$ is the full subposet of $\mc{A}_j$ consisting
of the objects with exactly $j$ components.  Thus it suffices
to show that $\abs{\mc{A}^{(k)}} \simeq \binom{M}{k}$.  For $k=0$,
this is trivial; now let $k$ be a positive integer.

Define the space $W$ to be the subspace of $\abs{\mc{A}^{(k)}}
\times \binom{M}{k}$ consisting of all points $(x,y)$ where $x$ is
in the interior of a nondegenerate $r$-simplex corresponding to
the string $V_0 \to \cdots \to V_r$, and each component of $V_r$
contains exactly one point of $y$.  We claim $W$ is an open set.
To see this, fix a point $(x,y)$ as above, and let $A$ be the subset 
of $\abs{\mc{A}^{(k)}}$ consisting of all interiors of simplices which
correspond to nondegenerate strings $U_0 \to \cdots \to U_s$ which
contain $V_0 \to \cdots \to V_r$ as a substring.  Note that $A$ is
an open set since it is a union of open cells, and whenever
an $s$-cell is in $A$, then every $(s+1)$-cell having it as a face
is also in $A$.  And now $(x,y) \in A \times V_r \subseteq W$,
so $W$ is open.

Since $W$ is open, it follows that the projections from $W$ to
each factor are almost locally trivial in the sense of \cite{Se}.
By a theorem of Segal \cite{Se}, if each fiber of
an almost locally trivial map is contractible, then the map
is a homotopy equivalence.  So if we can show each fiber of each
projection is contractible, then we have proved the lemma.

First, the fiber over any point of $\abs{\mc{A}^{(k)}}$ will be
diffeomorphic to $\reals^{kd}$.  Second, fix a $y \in
\binom{M}{k}$, and let $W_y$ be the fiber of the projection at $y$,
considered as a subspace of $\abs{\mc{A}^{(k)}}$.  Thus $W_y$ is the
union of the interiors of simplices corresponding to nondegenerate
strings $V_0 \to \cdots \to V_r$, where each component of $V_r$ contains
exactly one point from $y$.  Consider the subspace $W_y'$ of $W_y$ to be
the union of the interiors of only those simplices for which each
component of $V_0$ contains exactly one point from $y$.

Note that $W_y'$ is homeomorphic to the classifying space of
the full subposet of $\mc{A}^{(k)}$ consisting of all objects which
contain exactly one point of $y$ in each component.  Since
the objects of $\mc{A}_1$ form a basis for the topology of $M$,
this subposet is codirected, and hence its classifying space is
contractible by \cite{Qu}.  Thus $W_y' \simeq *$, and we finish the
proof of the lemma by showing that $W_y'$ is a deformation retract
of $W_y$.

Consider a nondegenerate string $V_0 \to \cdots \to V_r$
with each component of $V_r$ containing one point of $y$.
Let $q$ be the smallest index such that each component of
$V_q$ contains one point of $y$.  Let $\Delta$ be the
$r$-simplex in $\abs{\mc{A}^{(k)}}$ corresponding to
$V_0 \to \cdots \to V_r$.  Note that a point $x \in \Delta$ with
barycentric coordinates $(x_0,\ldots,x_r)$ is in $W_y$ iff
$x_q,\ldots,x_r$ are not all 0.  For each such $x \in \Delta \cap
W_y$, let $\conj{x}$ be the point in $\Delta \cap W_y'$ with coordinates
\begin{equation*}
  \frac{(0,\ldots,0,x_q,\ldots,x_r)}{x_q + \cdots + x_r}.
\end{equation*}
We now define a homotopy $H:W_y \times I \to W_y$ piecewise on (the
appropriate part of) each simplex $\Delta$ by the formula $H(x,t) =
(1-t)x + t\conj{x}$.

This is tentatively a deformation retraction of $W_y$ onto $W_y'$,
but we must still verify that $H$ is well-defined; the given formula
must agree on the intersection of simplices in $W_y$.
So, suppose $x\in\Delta$ has coordinates $(x_0,\ldots,x_r)$ and
corresponds to the string $V_0\to\cdots\to V_r$, and $x'\in\Delta'$
has coordinates $(x_0',\ldots,x'_{r'})$ and corresponds to the
string $V_0'\to\cdots\to V_{r'}'$.  Suppose further that $x$
and $x'$ represent the same point in $W_y$.  That is, their
corresponding strings share a (necessarily nonempty) maximal substring
\begin{equation*}
  V_{i_0} \to \cdots \to V_{i_s} = V'_{i'_0} \to \cdots \to V'_{i'_s}.
\end{equation*}
Furthermore, the only nonzero entries of $(x_0,\ldots,x_r)$ are
$x_{i_0},\ldots,x_{i_s}$, the only nonzero entries of $(x'_0,\ldots,
x'_{r'})$ are $x'_{i'_0},\ldots,x'_{i'_s}$, and $x_{i_j} = x_{i'_j}$
for all $0\leq j\leq s$.  Note that when defining the formula for
$\conj{x}$, we could have equivalently chosen $q$ to be the smallest
index for which $y\in V_q$ \emph{and} $x_q \neq 0$.  With this in
mind, we see that $\conj{x}$ and $\conj{x}'$ represent the same point,
and so we are done.
\end{proof}

\begin{notation}
For $p \geq 0$, let $\mc{A}_k(\mc{B}_k)_p = \mc{A}_k(\mc{B}_k)_p(M)$
be the category whose objects are strings of $p$ composable
morphisms in $\mc{B}_k$, $V_0 \to \cdots \to V_p$, and whose
morphisms are natural transformations of such diagrams whose
component maps all lie in $\mc{A}_k$.
\end{notation}

The next result is the analog of 3.6 and 3.7 in \cite{We}.

\begin{lemma}\label{thm:HofiberAkBk}
The homotopy fiber (over some 0-simplex $W$) of the map
\begin{equation*}
  \abs{\mc{A}_k(\mc{B}_k)_p(M)} \to \abs{\mc{A}_k(M)}
\end{equation*}
induced by the functor sending $(V_0 \to \cdots \to V_p)
\mapsto V_p$ is $\abs{\mc{A}_k(\mc{B}_k)_{p-1}(W)}$.
Furthermore, the functor $\abs{\mc{A}_k(\mc{B}_k)_p(-)}:\mc{O}(M)
\to\ctgy{Top}$ takes isotopy equivalences to homotopy equivalences.
\end{lemma}
\begin{proof}
The proof is by induction on $p$.  First consider the case
$p = 0$.  Note that $\abs{\mc{A}_k(\mc{B}_k)_0(U)}$ is just
$\abs{\mc{A}_k(U)}$, which by \cref{thm:NerveAk} is homotopy
equivalent to $\coprod_0^k \binom{U}{j}$.  Thus if $U\to U'$ is an
isotopy equivalence in $\mc{O}(M)$, then the induced map $\binom{U}{j}
\to \binom{U'}{j}$ is a homotopy equivalence for each $j$.

Now for $p > 0$, the Grothendieck construction gives us an
isomorphism of categories
\begin{align*}
\mc{A}_k(\mc{B}_k)_p(M) &\cong \mc{A}_k(M) \int
  \mc{A}_k(\mc{B}_k)_{p-1}(-)\\
V_0 \to \cdots \to V_p &\leftrightarrow (V_p, V_0 \to \cdots \to
  V_{p-1}).
\end{align*}
Combining this with Thomason's homotopy colimit theorem (\cref
{thm:ThomasonHocolim}), we get a homotopy equivalence
\begin{equation*}
\hocolim_{V\in\mc{A}_k(M)} \abs{\mc{A}_k(\mc{B}_k)_{p-1}(V)} \simeq
  \abs{\mc{A}_k(\mc{B}_k)_p(M)}.
\end{equation*}

The map in question then corresponds to the usual projection
of the homotopy colimit (which we take to be the usual
representation as the realization of the appropriate simplicial
replacement) to the nerve of the indexing category.
By the induction hypothesis and \cref{thm:QuillenDwyer}, this map
is a quasifibration, so the homotopy fiber we are interested in
has the same homotopy type as the actual fiber of this map.  By
the proof of \cref{thm:RestrictedHolimEquiv}, this fiber can be taken
to be $\abs{\mc{A}_k(\mc{B}_k)_{p-1}(W)}$, as we needed to show.

For the second part of the lemma, let $V \to V'$ be an isotopy
equivalence in $\mc{O}(M)$.  We have a map of quasifibration sequences
\begin{equation*}
\xymatrix{
\abs{\mc{A}_k(\mc{B}_k)_{p-1}(W)} \ar[r] \ar[d] &
  \abs{\mc{A}_k(\mc{B}_k)_p(V)} \ar[r] \ar[d] &
  \abs{\mc{A}_k(V)} \ar[d] \\
\abs{\mc{A}_k(\mc{B}_k)_{p-1}(W)} \ar[r] &
  \abs{\mc{A}_k(\mc{B}_k)_p(V')} \ar[r] &
  \abs{\mc{A}_k(V')} \\
}\end{equation*}
The lefthand vertical map can be taken to be the identity,
and the righthand vertical map is a homotopy equivalence as
mentioned above.  Therefore, so is the middle vertical map.
\end{proof}

Let $F:\mc{B}_k\to\ctgy{Top}$ be an isotopy cofunctor.
Define $F_p:\mc{A}_k(\mc{B}_k)_p\to\ctgy{Top}$ by
$F_p(U_0 \to \cdots \to U_p) = F_p(U_0)$.  Define $F_p^!:\mc{O}\to
\ctgy{Top}$ as:
\begin{equation*}
F_p^!(V) = \holim_{\mc{A}_k(\mc{B}_k)_p(V)} F_p.
\end{equation*}

\begin{lemma}\label{thm:FpBangGood}
$F_p^!$ is a good cofunctor.
\end{lemma}
\begin{proof}
For part (a) of goodness, let $V\to V'$ be an isotopy equivalence
in $\mc{O}$.  Since $F_p$ takes all morphisms to homotopy
equivalences, then by \cref{thm:RestrictedHolimEquiv} it suffices to
show that the inclusion $\abs{\mc{A}_k(\mc{B}_k)_p(V)}
\hookrightarrow \abs{\mc{A}_k(\mc{B}_k)_p(V')}$ is a homotopy
equivalence, and it is by \cref{thm:HofiberAkBk}.

For part (b) of goodness, let $V_0 \to V_1 \to \cdots$ be a string
in $\mc{O}$.  We need to show that the map
\begin{equation*}
F_p^!\parens{\bigcup_i V_i} \to \holim_i F_p^!(V_i)
\end{equation*}
is a homotopy equivalence.  Rewrite and factor this as
\begin{equation*}
\holim_{\mc{A}_k(\mc{B}_k)_p(\bigcup V_i)} F_p \to
  \holim_{\bigcup \mc{A}_k(\mc{B}_k)_p(V_i)} F_p \to
  \holim_i \holim_{\mc{A}_k(\mc{B}_k)_p(V_i)} F_p.
\end{equation*}
The right map is a homotopy equivalence by \cref{thm:FilteredHolim}.
Furthermore, if we can show that the inclusion
\begin{equation*}
  \abs{\bigcup_i \mc{A}_k(\mc{B}_k)_p(V_i)} \hookrightarrow
  \abs{\mc{A}_k(\mc{B}_k)_p\parens{\bigcup_i V_i}}
\end{equation*}
is a homotopy equivalence, then \cref{thm:RestrictedHolimEquiv}
would imply that the left map is a homotopy equivalence as well.

Note that the nerve of a union is the same as the union of
nerves.  First, if $p = 0$, we have the inclusion
\begin{equation*}
  \bigcup_i\abs{\mc{A}_k(V_i)} \hookrightarrow
  \abs{\mc{A}_k\parens{\bigcup_i V_i}},
\end{equation*}
which by \cref{thm:NerveAk} is the same as the map
\begin{equation*}
  \bigcup_i \coprod_{j=0}^k \binom{V_i}{j} \hookrightarrow
  \coprod_{j=0}^k \binom{\bigcup_i V_i}{j}.
\end{equation*}
This is actually a homeomorphism since the union and
coproduct commute, and since the configuration spaces in
question are selecting only a finite number of points.
For the $p > 0$ case, consider the inclusion of homotopy
fibration sequences over some base point $W$ in some
$\mc{A}_k(V_i)$.
\begin{equation*}
\xymatrix{
\abs{\mc{A}_k(\mc{B}_k)_{p-1}(W)} \ar[r] \ar@2{-}[d] &
  \bigcup_i \abs{\mc{A}_k(\mc{B}_k)_p(V_i)} \ar[r] \ar[d] &
  \bigcup_i \abs{\mc{A}_k(V_i)} \ar[d] \\
\abs{\mc{A}_k(\mc{B}_k)_{p-1}(W)} \ar[r] &
  \abs{\mc{A}_k(\mc{B}_k)_p\parens{\bigcup_i V_i}} \ar[r] &
  \abs{\mc{A}_k\parens{\bigcup_i V_i}} \\
}\end{equation*}
The righthand map is a homotopy equivalence by the $p = 0$
case above, so the middle map is as well.
\end{proof}

\begin{lemma}\label{thm:TotFpBangToFBang}
For any $V \in \mc{O}$, the projection
\begin{equation*}
  \Tot(p\mapsto F_p^!(V)) \to F^!(V)
\end{equation*}
is a homotopy equivalence.
\end{lemma}
\begin{proof}
Let $(\mc{A}_k)_q\mc{B}_k = (\mc{A}_k)_q\mc{B}_k(M)$ be the category
whose objects are strings of $q$ composable morphisms in $\mc{A}_k$,
and whose morphisms are natural transformations of such diagrams with
component maps in $\mc{B}_k$.  Now, the domain of the projection
can be thought of as a totalization of a totalization (of the
cosimplicial replacement of $F_p$).  We can switch the order of
the totalizations and rewrite the domain as $\Tot(q\mapsto
\wh{F}_q^!(V))$, where $\wh{F}_q:(\mc{A}_k)_q\mc{B}_k(V) \to
\ctgy{Top}$ is the functor sending $U_0 \to \cdots \to U_q
\mapsto F(U_0)$.  Thus, if the map $\wh{F}_q^!(V) \to F^!(V)$
is a homotopy equivalence for each $q$, then so will the original
map be.
\par
We can write this map as the map
\begin{equation*}
\holim_{(\mc{A}_k)_q\mc{B}_k(V)} \wh{F}_q \to \holim_{\mc{B}_k(V)} F
\end{equation*}
which is induced by the inclusion functor $J:\mc{B}_k(V) \to
(\mc{A}_k)_q\mc{B}_k(V)$ sending $U \mapsto U \to \cdots \to U$.
We claim that $J$ is homotopy terminal.  That is, for any
$U_0 \to \cdots \to U_q \in (\mc{A}_k)_q\mc{B}_k(V)$, the
comma category $(U_0 \to \cdots \to U_q \downarrow J)$ is
nonempty and contractible.  This follows since the comma category
has an intial object, namely $U_q \to \cdots \to U_q$.  Therefore,
the induced map of homotopy limits is a homotopy equivalence, and
this proves the lemma.
\end{proof}

\begin{thm}\label{thm:FBangGood}
$F^!$ is a good cofunctor.
\end{thm}
\begin{proof}
Let $V\to V'$ be an isotopy equivalence in $\mc{O}$.  By the
previous lemma, the horizontal maps in the square
\begin{equation*}
\xymatrix{
\Tot(p\mapsto F_p^!(V)) \ar[r]^-\simeq & F^!(V) \\
\Tot(p\mapsto F_p^!(V')) \ar[r]^-\simeq \ar[u] &
  F^!(V') \ar[u] \\
}
\end{equation*}
are homotopy equivalences, so showing the righthand map is a
homotopy equivalence for part (a) of goodness is equivalent to
showing so for for the lefthand map.  In codegree $p$,
the lefthand map is just the map $F_p^!(V') \to F_p^!(V)$, which
is a homotopy equivalence since $F_p^!$ is good.  So the overall
map of totalizations is a homotopy equivalence since the underlying
map of cosimplicial spaces is a homotopy equivalence in each
codegree.

For part (b) of goodness, let $V_0 \to V_1 \to \cdots$ be a
string in $\mc{O}$.  Similiary, to show that the map
\begin{equation*}
F^!\parens{\bigcup_i V_i} \to \holim_i F^!(V_i)
\end{equation*}
is a homotopy equivalence, it suffices to show that the map
\begin{equation*}
\Tot\bracks{p \mapsto F_p^!\parens{\bigcup_i V_i}} \to
  \holim_i \Tot\bracks{p \mapsto F_p^!(V_i)}
\end{equation*}
is one.  But the homotopy limit and totalization commute, so
this is really a map of totalizations of cosimplicial spaces
which in codegree $p$ is the map
\begin{equation*}
F_p^!\parens{\bigcup_i V_i} \to \holim_i F_p^!(V_i).
\end{equation*}
And again, this is a homotopy equivalence since $F_p^!$ is good.
\end{proof}

We come to our main result.
\begin{thm}\label{thm:BkRefinement}
Suppose $\braces{\mc{B}_k}$ and $\braces{\mc{B}_k'}$ are two
choices of special open sets with $\mc{B}_1' \subseteq \mc{B}_1$
(which implies that $\mc{B}_k' \subseteq \mc{B}_k$ as well).
Let $F:\mc{B}_k\to\ctgy{Top}$ be an isotopy cofunctor, and
let $G$ denote its restriction to $\mc{B}_k'$.  Then the map
$F^!(V) \to G^!(V)$ is a homotopy equivalence for all $V \in \mc{O}$.
\end{thm}
\begin{proof}
Let $V \in \mc{O}$.  By a similar argument as before, it suffices
to check that the map $F_p^!(V) \to G_p^!(V)$ is a homotopy
equivalence for all $p$.  And by \cref{thm:RestrictedHolimEquiv}, it is
enough to check that the inclusion $\abs{\mc{A}'_k(\mc{B}'_k)_p(V)}
\hookrightarrow \abs{\mc{A}_k(\mc{B}_k)_p(V)}$ is a homotopy
equivalence. Note that if $p = 0$, then this is really the inclusion
$\abs{\mc{A}_k'(V)} \hookrightarrow \abs{\mc{A}_k(V)}$, and
this is a homotopy equivalence by \cref{thm:NerveAk}.
If $p > 0$, choose a basepoint $W \in \abs{\mc{A}_k'(V)}$ and
use \cref{thm:HofiberAkBk} to get a diagram of homotopy
fibration sequences
\begin{equation*}
\xymatrix
{
  \abs{\mc{A}_k'(\mc{B}_k')_{p-1}(W)} \ar[r] \ar[d] &
    \abs{\mc{A}_k'(\mc{B}_k')_p(V)} \ar[r] \ar[d] &
    \abs{\mc{A}_k'(V)} \ar[d] \\
  \abs{\mc{A}_k(\mc{B}_k)_{p-1}(W)} \ar[r] &
    \abs{\mc{A}_k(\mc{B}_k)_p(V)} \ar[r] &
    \abs{\mc{A}_k(V)} \\
}
\end{equation*}
The righthand vertical map is again the $p = 0$ case, so is
a homotopy equivalence.  The lefthand vertical map can be
assumed to be a homotopy equivalence by induction on $p$.
Therefore, the middle vertical map is a homotopy equivalence
as well, as we needed to show.
\end{proof}

As an immediate corollary, we get a strengthening of Weiss'
construction of polynomial cofunctors (\cref{thm:WeissFBangPoly}).
\begin{cor}\label{thm:FBangPoly}
Let $F$ be an isotopy cofunctor from $\mc{B}_k\to\ctgy{Top}$.
Then $F^!$ is polynomial of degree $\leq k$.
\end{cor}
\begin{proof}
This follows from \cref{thm:WeissFBangPoly} and \cref{thm:BkRefinement}.
\end{proof}

\bibliography{SOS}
\end{document}